%% file: DegenLift.3.tex
\documentclass[oneside,english]{amsart}
\usepackage{mathptmx}
\usepackage[T1]{fontenc}
\usepackage[latin9]{inputenc}
\usepackage{fancyhdr}
\usepackage{babel}
\usepackage{mathrsfs}
\usepackage{url}
\usepackage{amsthm}
\usepackage{amssymb}
\usepackage[unicode=true,pdfusetitle,
 bookmarks=false,
 breaklinks=false,pdfborder={0 0 1},backref=false,colorlinks=false]
 {hyperref}

\makeatletter
\numberwithin{equation}{section}
\numberwithin{figure}{section}
\theoremstyle{plain}
\newtheorem{thm}{\protect\theoremname}
  \theoremstyle{plain}
  \newtheorem{conjecture}[thm]{\protect\conjecturename}
  \theoremstyle{definition}
  \newtheorem{problem}[thm]{\protect\problemname}
  \theoremstyle{plain}
  \newtheorem{lem}[thm]{\protect\lemmaname}
  \theoremstyle{plain}
  \newtheorem{cor}[thm]{\protect\corollaryname}
  \theoremstyle{definition}
  \newtheorem{defn}[thm]{\protect\definitionname}
  \theoremstyle{remark}
  \newtheorem{notation}[thm]{\protect\notationname}
  \theoremstyle{plain}
  \newtheorem{fact}[thm]{\protect\factname}
  \theoremstyle{remark}
  \newtheorem{rem}[thm]{\protect\remarkname}
  \theoremstyle{plain}
  \newtheorem{prop}[thm]{\protect\propositionname}

\include{mydefs}

\newcommand\mb{\bar\mu}
\newcommand\mbn{{\mb}_n}
\newcommand\mbi{{\mb}_\infty}
\newcommand\mn{\mu_n}
\newcommand\mi{\mu_\infty}
\newcommand\sn{\sigma_n}

\newcommand\nt{\tilde{\nu}}
\newcommand\ntn{\nt_n}
\newcommand\nti{\nt_\infty}

\newcommand\mk{M\bs K}
\newcommand\Vc{C(\mk)}
\newcommand\VK{V_K}
\newcommand\VKh{\hat{V}_K}

\newcommand\Test{\Cic(X)_K}
\newcommand\Distr{\Test'}

\newcommand\Dr{\Delta^\mathrm{r}}

\newcommand\clC{\overline{\sC}}
\newcommand\opC{{\sC}}
\newcommand\cCw{\clC_w}
\newcommand\oCw{\opC_w}

\newcommand\sequ[1]{\left\{ {#1}_n \right\}_{n=1}^{\infty}}

\newcommand\gK{\left(\lieg,K\right)}

\newcommand\rR{\Delta(\lieg:\liea)}
\newcommand\CR{\Delta(\lieg_\C:\lieh_\C)}
\newcommand\rW{W(\lieg:\liea)}
\newcommand\CW{W(\lieg_\C:\lieh_\C)}

\makeatother

  \providecommand{\conjecturename}{Conjecture}
  \providecommand{\corollaryname}{Corollary}
  \providecommand{\definitionname}{Definition}
  \providecommand{\factname}{Fact}
  \providecommand{\lemmaname}{Lemma}
  \providecommand{\notationname}{Notation}
  \providecommand{\problemname}{Problem}
  \providecommand{\propositionname}{Proposition}
  \providecommand{\remarkname}{Remark}
\providecommand{\theoremname}{Theorem}

\begin{document}

\title{Quantum unique ergodicity on locally symmetric spaces: the degenerate
lift}

\author{Lior Silberman}

\address{Department of Mathematics, University of British Columbia, Vancouver\ \ BC\ ~V6T
1Z2, Canada}

\email{\url{lior@math.ubc.ca}}



\date{\today}
\begin{abstract}
Given a measure $\mbi$ on a locally symmetric space $Y=\Gamma\backslash G/K$,
obtained as a weak-{*} limit of probability measures associated to
eigenfunctions of the ring of invariant differential operators, we
construct a measure $\mi$ on the homogeneous space $X=\Gamma\backslash G$
which lifts $\mbi$ and which is invariant by a connected subgroup
$A_{1}\subset A$ of positive dimension, where $G=NAK$ is an Iwasawa
decomposition. If the functions are, in addition, eigenfunctions of
the Hecke operators, then $\mi$ is also the limit of measures associated
to Hecke eigenfunctions on $X$. This generalizes previous results
of the author and A.\ Venkatesh to the case of {}``degenerate''
limiting spectral parameters.
\end{abstract}
\maketitle

\section{Introduction}

In the work of the author with A.\ Venkatesh \cite{SilbermanVenkatesh:SQUE_Lift}
we investigated the asymptotic behaviour of eigenfunctions on high-rank
locally symmetric spaces, under the assumption that the spectral parameters
(see below) were \emph{non-degenerate}, in that their imaginary parts
were located away from the walls of the Weyl chamber (in particular,
this forced the spectral parameters to lie on the unitary axis). This
paper removes this assumption, at the cost of a weaker invariance
statement for the limiting measures. The main extra ingredient is
a simple calculation in the {}``compact'' model of induced representation
for semisimple Lie groups.

\subsection{The problem of Quantum Unique Ergodicity; statement of the result}

Let $Y$ be a (compact) Riemannian manifold. To a non-zero eigenfunction
$\psi_{n}$ of the Laplace-Beltrami operator $\Lap$ with eigenvalue
$-\lambda_{n}$ we attach the probability measure 
\[
\mbn(\varphi)=\frac{1}{\norm{\psi_{n}}^{2}}\int_{Y}\left|\psi_{n}(y)\right|^{2}\varphi(y)dy\,.
\]
Classifying the possible limits (in the weak-{*} sense) of sequences
$\sequ{\mb}$ where $\lambda_{n}\to\infty$ is known as the problem
of {}``Quantum Unique Ergodicity'' (specifically, {}``QUE on $Y$'').
Nearly all attacks on this problem begin by associating to each measure
$\mbn$ a distribution ({}``microlocal lift'') $\mn$ on the unit
cotangent bundle $S^{*}Y$ which projects to $\mbn$ on $Y$, in such
a way that any weak-{*} limit of the $\mn$ is a probability measure,
invariant under the geodesic flow on $S^{*}Y$. This construction
(due to Schnirel'man, Zelditch and Colin de Verdière, \cite{Schnirelman:Avg_QUE,Zelditch:SL2_Lift_I,CdV:Avg_QUE})
leads to a reformulation of the problem ({}``QUE on $S^{*}Y$''),
where one seeks to classify the weak-{*} limits of sequences such
as $\sequ{\mu}$. Now results from dynamical systems concerning measures
invariant under the geodesic flow can be brought to bear. In particular,
under the very general assumption that the geodesic flow on $S^{*}Y$
is ergodic, it was shown by these authors that the Riemannian volume
measure on $S^{*}Y$ is always a limit measure for some sequence of
eigenfunctions (hence its projection, the Riemannian volume on $Y$,
is always a limit of a sequence of measures $\mbn$). The most spectacular
realization of this approach to QUE is in the work of Lindenstrauss
\cite{Lindenstrauss:SL2_QUE}. There it is shown that on congruence
hyperbolic surfaces and for eigenfunctions $\psi_{n}$ which are also
eigenfunctions of the so-called Hecke operators the Riemannian volume
is the \emph{only} limiting measure%
\footnote{For non-compact surfaces this statement requires the result of \cite{Sound:EscapeOfMass}.%
}. In fact, Rudnick-Sarnak \cite{RudnickSarnak:Conj_QUE} conjecture
that this phenomenon (uniqueness of the limit) holds for all manifolds
$Y$ of (possibly variable) negative sectional curvature. Results
in that level of generality have also appeared recently, starting
with the breakthrough of \cite{Anantharaman:QUE_Ent}.

In this paper we consider a technical aspect of the problem on locally
symmetric spaces $Y=\Gamma\backslash G/K$ of non-compact type. Here
$G$ is a semisimple Lie group with finite center, $K$ a maximal
compact subgroup and $\Gamma<G$ a lattice (thus $Y$ is of finite
volume but not necessarily compact). On such spaces there is a natural
commutative algebra of differential operators containing the Laplace-Beltrami
operator, and it is better to consider joint eigenfunctions of this
algebra. This is the algebra of $G$-invariant differential operators
on $G/K$, which may be identified with the center of the universal
enveloping algebra of the Lie algebra of $G$. Accordingly, let $\psi_{n}\in L^{2}(Y)$
be joint eigenfunctions of this algebra. The approach of microlocal
analysis applies to this setting as well (see \cite{AnantharamanSilberman:SQUE-density-preprint}),
lifting measures to distributions on $S^{*}Y$, but in fact limits
of these measures are supported on singular subsets there, isomorphic
to submanifolds of the form $\Gamma\backslash G/M_{1}$ for compact
subgroups $M_{1}$. Here we directly construct a lift to this space.
Moreover, in the congruence setting it is desirable to have the lift
be manifestly equivariant with respect to the action of the Hecke
algebra. In the paper \cite{SilbermanVenkatesh:SQUE_Lift} this was
done under a genericity assumption ({}``non-degeneracy'') -- that
the sequence of spectral parameters $\nu_{n}\in\acs$ (here $\liea=\Lie(A)$
where $G=NAK$ is an Iwasawa decomposition) associated to the $\psi_{n}$
be contained in a proper subcone of the open Weyl chamber in $i\ars$.
Under that assumption, and weak-{*} limit $\mbi$ of a sequence as
above was seen to be the projection of an $A$-invariant positive
measure $\mi$ on $X$. In this paper the non-degeneracy assumption
is removed, giving our main result:
\begin{thm}
\label{mainthm:lift}Assume $\mbn\wklim{n\to\infty}\mbi$. Then there
exists a non-trivial connected subgroup $A_{1}\subset A$ and an $A_{1}$-invariant
positive measure $\sigma_{\infty}$ on $X$ projecting to $\mbi$.

In more detail, let $\Test$ be the space of right $K$-finite smooth
functions of compact support on $X=\Gamma\backslash G$. By a \emph{distribution}
we shall mean an element of its algebraic dual. Then, after passing
to a subsequence, we obtain distributions $\mn\in\Distr$ and functions
$\tilde{\psi}_{n}\in L^{2}(X)$ such that:
\begin{enumerate}
\item \label{enu: m-Lift}(\emph{Lift}) The distributions $\mn$ project
to the measures $\mbn$ on $Y$. In other words, for $\varphi\in\Cic(Y)$
we have $\mn(\varphi)=\mbn(\varphi)$.
\item \label{enu: m-positive}Let \emph{$\sn$} be the measure on $X$ such
that \emph{$d\sn(x)=\left|\tilde{\psi}_{n}(x)\right|^{2}dx$}. Then:

\begin{enumerate}
\item (\emph{Positivity}) $\sequ{\sigma}$ converges weak-{*} to a measure
$\sigma_{\infty}$ on $X$, necessarily a positive measure of total
mass $\leq1$.
\item (\emph{Consistency}) For any \emph{$\varphi\in\Cic(X)_{K}$}, $\left|\sn(\varphi)-\mn(\varphi)\right|\to0$
as $n\to\infty$.
\end{enumerate}
\item \label{enu: m-A-inv}(\emph{Invariance}) Let the normalized spectral
parameters%
\footnote{For $G$ simple, these are $\frac{\nu_{n}}{\norm{\nu_{n}}}$. For
$G$ semisimple see the discussion in \cite[\S5.1]{SilbermanVenkatesh:SQUE_Lift}%
} $\tilde{\nu}_{n}$ converge to a limiting parameter $\tilde{\nu}_{\infty}$
in the closed positive Weyl chamber of $i\ars$. Then $\mi$ is invariant
by $A_{1}Z_{K}(A_{1})$, where $A_{1}\subset A$ is the set of elements
fixed by the stabilizer $W_{1}=\Stab_{W}(\tilde{\nu}_{\infty})$.
\item \label{enu: m-equiv}(\emph{Equivariance}) $\tilde{\psi}_{n}$ belong
to the irreducible subrepresentation of $G$ in $L^{2}(Y)$ generated
by $\psi_{n}$. In particular, if $\mathcal{H}$ is a commutative
algebra of bounded operators on $L^{2}(X)$ which commute with the
$G$-action and $\psi_{n}$ is a joint eigenfunction of $\mathcal{H}$
then so is $\tilde{\psi}_{n}$, with the same eigenvalues.
\end{enumerate}
\end{thm}

\subsection{Sketch of the proof}

As can be expected, we shall trace a path very similar to that of
the previous work. Choose a pair of functions, one from the irreducible
subrepresentation of $L^{2}(X)$ generated by $\psi_{n}$ and one
from its dual. Now integrating a function on $X$ against the product
of these two functions defines a measure there ($\mbn$ is a special
case of this construction), and we will study limits of this larger
family of measures. We construct an asymptotic calculus for these
measures by uniformizing the representation via the compact picture
of a principal series representations induced from a potentially non-unitary
character. A prerequisite for taking limits in this setting is the
following \emph{a-priori} bound on these measures \wrt the uniformization,
which is the key ingredient that was not available during the writing
of \cite{SilbermanVenkatesh:SQUE_Lift}. 
\begin{thm}
\label{mainthm: intops}Let $\left(\pi,V_{\pi}\right)\in\hat{G}$
be spherical, and let $R\colon\left(I_{\nu},V_{K}\right)\to\left(\pi,V_{\pi}\right)$
be an%
\footnote{Such $R$ always exist.%
} intertwining operator with the real part of $\nu\in\acs$ in the
closed positive chamber $\cC$, normalized such that $\left\Vert R(\varphi_{0})\right\Vert _{V_{\pi}}=1$,
where $\varphi_{0}\in V_{K}$ is the constant function $1$. We then
have $\left\Vert R(f)\right\Vert _{V_{\pi}}\leq\left\Vert f\right\Vert _{L^{2}(K)}$
for any $f\in V_{K}$.
\end{thm}
Surprisingly, we could not find this useful fact in the literature.
It is proved in Section \ref{sec: Int-bds} as a consequence of the
rationality of $K$-finite matrix coefficients by bounding the analytical
continuation of the normalized intertwining operators $\tilde{A}(\nu;w)\colon(I_{\nu},V_{K})\to(I_{w\nu},V_{K})$
associated to elements $w$ of the Weyl group $W$.

With this bound in hand we extend the asymptotic calculus of \cite{SilbermanVenkatesh:SQUE_Lift}
to our setting. We construct the distributions $\mn$ in Section \ref{sub: lift-one}
(see Definition \ref{def: Lift}). Integration by parts gives the
measures $\sn$ and establishes their properties (Corollaries \ref{cor: positive}
and \ref{cor: equiv}). Finally, in Section \ref{sub: A-inv} we obtain
the desired invariance property.

\subsection{A measure rigidity problem on locally symmetric spaces}

An introduction to the relations between the general problem of Quantum
Unique Ergodicity and the cases of manifolds of negative curvature
and locally symmetric spaces of non-positive curvature may be found
in the paper \cite{SilbermanVenkatesh:SQUE_Lift}. We consider here
only the latter case, where again we have the Conjecture
\begin{conjecture}
(Sarnak) The sequence $\sequ{\mb}$ converges weak-{*} to the normalized
volume measure $\frac{d\vol_{Y}}{\vol(Y)}$.
\end{conjecture}
We recall the strategy pioneered by Lindenstrauss, which applies for
a lattice $\Gamma$ for which there exists a large algebra $\mathcal{H}$
of bounded normal operators on $L^{2}(X)$, commuting with the $G$-action.
We then consider a sequence of joint eigenfunctions of both the differential
operators and of $\mathcal{H}$, and assume the associated measures
$\mbn$ converge to a measure $\mbi$. 
\begin{enumerate}
\item \emph{Lift}: Passing to a subsequence, lift $\mbi$ to a positive
measure $\mi$ on $X$ which projects to $\mbi$ under averaging by
$K$ and is invariant under a subgroup $H<G$, in a way which respects
the $\mathcal{H}$-action.
\item \emph{Extra smoothness}: Using the geometry of the action of $\mathcal{H}$,
show that any measure $\mu_{\infty}$ thus obtained is not too singular
(for example, that the dimension of its support must be strictly larger
than that of $H$).
\item \emph{Measure rigidity}: Using classification results for $H$-invariant
measures on $X$, show that the additional information of Step (2)
forces $\mu_{\infty}$ to be a $G$-invariant measure on $X$.
\end{enumerate}
The result of this paper extend Step (1) of the strategy to the degenerate
case and the methods used for Step (2) in \cite{LindenstraussBourgain:SL2_Ent}
and \cite{SilbermanVenkatesh:AQUE_Ent_preprint} only use the Hecke
operators. Unfortunately, current higher-rank measure classification
results (such as the one in \cite{EKL:SLn_Rigid}, used for Step (3)
in \cite{SilbermanVenkatesh:AQUE_Ent_preprint}) do not readily generalize
to the case of $A_{1}$-invariant measures; in this context see the
counter-example {[}\cite{Maucourant:MargulisCounter}. However, we
are not considering a general $A_{1}$-invariant measure, so the natural
question from our point of the view is the following. It should be
compared with Lindenstrauss's rank $1$ measure classification Theorem
\cite[Thm.\ 1.1]{Lindenstrauss:SL2_QUE}.
\begin{problem}
Let $\Gamma<G$ be a congruence lattice associated to a $\Q$-structure
on $G$ and let $A_{1}\subset A$ be a non-trivial one-parameter subgroup
fixed by a subgroup of the Weyl group. Let $\tilde{\psi}_{n}\in L^{2}(X)$
be eigenfunctions of the Hecke operators on $X=\Gamma\backslash G$
such that their associated probability measures $\sigma_{n}$ converge
weak-{*} to an $A_{1}$-invariant measure $\sigma_{\infty}$. Is it
true that $\sigma_{\infty}$ is then a (continuous) linear combination
of algebraic measures on $X$?
\end{problem}

\section{Notation and Preliminaries}

\subsection{Structure theory -- real groups}

Let $G$ be a connected almost simple Lie group%
\footnote{The results of this paper hold (with natural modifications) for reductive
$G$. The details may be found in \cite[\S5.1]{SilbermanVenkatesh:SQUE_Lift}.%
}, $\lieg=\Lie(G)$ its Lie algebra. Let $\Theta$ be a Cartan involution
for $G$, $\theta$ the differential of $\Theta$ at the identity
and let $\lieg=\liep\oplus\liek$ be the associated polar decomposition.
We fix a maximal Abelian subalgebra $\liea\subset\liep$. Its dimension
is the (real) rank $\rk G$.

The dual vector space to $\liea$ will be denoted $\ars$, and will
be distinguished from the complexification $\acs\eqdef\ars\otimes_{\R}\C$.
For $\alpha\in\ars$ set $\mathfrak{g}_{\alpha}=\left\{ X\in\lieg\mid\forall H\in\liea:[H,X]=\alpha(H)X\right\} $.
Let $\Delta=\Delta(\lieg:\liea)$ denote the set of \emph{roots} (the
non-zero $\alpha\in\ars$ such that $\lieg_{\alpha}\neq0$). Then
$\lieg=\lieg_{0}\oplus\bigoplus_{\alpha\in\Delta}\lieg_{\alpha}$,
and $\lieg_{0}=\liea\oplus\liem$ where $\liem=Z_{\liek}(\liea)$.
For $\alpha\in\ars$ we set $p_{\alpha}=\dim\lieg_{\alpha}$, $q_{\alpha}=\dim\lieg_{2\alpha}.$

The Killing form $B$ induces a positive-definite pairing $\left\langle X,Y\right\rangle =-B(X,\theta Y)$
on $\lieg$ which remains non-degenerate when restricted to $\liea$.
We identify $\liea$ and $\ars$ via this pairing, giving us a non-degenerate
pairing $\left\langle \cdot,\cdot\right\rangle $ on $\ars$ and letting
$H_{\alpha}\in\liea$ denote the element corresponding to $\alpha\in\Delta$.
With this Euclidean structure on $\ars$ the subset $\Delta$ is a
root system, and we denote its Weyl group by $\rW$. A root $\alpha\in\Delta$
is \emph{reduced} if $\frac{1}{2}\alpha\notin\Delta$. The set of
reduced roots $\Dr\subset\Delta$ is a root system as well. To $w\in W$
we associate the subset $\Phi_{w}=\Dr\cap\Delta^{+}\cap w^{-1}\Delta^{-}$
of positive reduced roots $\beta$ such that $w\beta$ is negative.

We fix a simple system $\Pi\subset\Delta$, giving us a notion of
positivity, and let $\Delta^{+}$ ($\Delta^{-}$) denote the set of
positive (negative) roots, $\rho=\frac{1}{2}\sum_{\alpha\in\Delta^{+}}p_{\alpha}\alpha\in\ars$.
For $\beta\in\Dr$ and $\nu\in\acs$ we set $\nu_{\beta}=\frac{2\left\langle \nu,\beta\right\rangle }{\left\langle \beta,\beta\right\rangle }$.
Then 
\[
\opC=\left\{ \nu\in\ars\mid\forall\beta\in\Pi:\nu_{\beta}>0\right\} 
\]
is the open positive Weyl chamber. Its closure will be denoted $\cC$.
We will also consider the open domain 
\[
\Omega=\opC+i\ars=\left\{ \nu\in\acs\vert\Re(\nu)\in\opC\right\} 
\]
 and its closure $\bar{\Omega}$. More generally, for $w\in W$ we
set
\[
\oCw=\left\{ \nu\in\ars\mid\forall\beta\in\Phi_{w}:\nu_{\beta}>0\right\} 
\]
leading in the same fashion to $\cCw\subset\ars$ and $\Omega_{w}\subset\bar{\Omega}_{w}\subset\acs$.

Returning to the Lie algebra we set $\lien=\oplus_{\alpha\in\Delta^{+}}\lieg_{\alpha}$,
$\lienb=\theta\lien=\oplus_{\alpha\in\Delta^{-}}\lieg_{\alpha}$ and
obtain the Iwasawa decomposition $\lieg=\lien\oplus\liea\oplus\liek$.
On the group level we set $K=\left\{ g\in G\mid\Theta(g)=g\right\} $,
$A=\exp\liea$, $N=\exp\lien$, $\bar{N}=\exp\lienb$. These are closed
subgroups with Lie algebras $\liek,\liea,\lien,\lienb$ respectively:
$K$ is a maximal compact subgroup, $A$ a maximal diagonalizable
subgroup and $N$ a maximal unipotent subgroup. With these we have
the Iwasawa decomposition $G=NAK$. Another important subgroup is
$M=Z_{K}(\liea)$ which normalizes $N,\bar{N}$. $M$ is not necessarily
connected, but $\liem=\Lie(M)$ holds, and $B=NAM$ is the Borel subgroup.
The action of $W=N_{K}(\liea)/Z_{K}(\liea)$ on $\ars$ gives an isomorphism
of $W$ and the algebraic Weyl group $\rW$ defined above.

Let $dk$ be a probability Haar measure on $K$, $da$, $dn$ Haar
measures on $A$ and $N$. Then $dn\cdot a^{2\rho}da\cdot dk$ is
a Haar measure on $G$. The linear functional $f\mapsto\int_{K}f(k)dk$
on the space $\mathcal{F}^{\rho}=\left\{ f\in C(G)\colon f(nag)=a^{2\rho}f(g)\right\} $
is right $G$-invariant.

\subsection{Complexification}

Let $\lieb$ be a maximal torus in the compact Lie algebra $\liem$,
$\lieh=\liea\oplus\lieb$. Then $\lieh$ is a maximal Abelian semisimple
subalgebra of $\lieg$, that is a Cartan subalgebra.

$\lieg_{\C}$ is a complex semisimple Lie algebra of which $\lieh_{\C}$
is a Cartan subalgebra. We let $\CR$ denote the associated root system,
$\CW$ its Weyl group. The restriction of any $\alpha\in\CR$ to $\liea$
is either a root of $\lieg$ or zero. We fix a notion of positivity
on $\CR$ compatible with our choice for $\rR$, and let $\rho_{\lieh}\in\lieh_{\C}^{*}$
denote half the sum of the positive roots in $\CR$. Once $\rho_{\lieh}$
makes its appearance we shall use $\rho_{\liea}$ for $\rho$ defined
before.

The image of $N_{K}(\lieh)$ in $\CW$ is $\tilde{W}=N_{K}(\lieh)/Z_{M}(\lieb)$,
since any $k\in N{}_{K}(\lieh)$ must normalize $\liea$, $\lieb$
separately.
\begin{lem}
$W(\liem\colon\lieb)\isom N_{M}(\lieb)/Z_{M}(\lieb)$ is normal in
$\tilde{W}$; the quotient is naturally isomorphic to $\rW$.\end{lem}
\begin{proof}
That $N_{M}(\lieb)=N_{K}(\lieh)\cap Z_{K}(\liea)$ gives the first
assertion, and shows that the quotient embeds in $\rW$ since $N_{K}(\lieh)\subset N_{K}(\liea)$.
To show that the embedding is surjective let $w\in N_{K}(\liea)$
and consider $\Ad(w)\lieb$. This is the Lie algebra of a maximal
torus of $M$ ($\Ad(w)$ is an automorphism of $M$), hence conjugate
to $\lieb$ in $M$. In other words, there exists $m\in M$ such that
$\Ad(w)\lieb=\Ad(m)\lieb$ and hence $m^{-1}w\in N_{K}(\lieb)$. This
element also normalizes $\liea$, and hence $wM\in W$ has a representative
in $N_{K}(\lieh)$.\end{proof}
\begin{cor}
\label{cor: complex}Under the identification $\acs\isom\left\{ \nu\in\lieh_{\C}^{*}\mid\nu\upharpoonright_{\lieb}\equiv0\right\} $
(dual to the identification $\liea\isom\lieh/\lieb$) the group $\tilde{W}\subset\CW$
acts on $\acs$ via its quotient map to $W$.
\end{cor}
We let $U(\lieg_{\C})$ denote the universal enveloping algebra of
the complexification of $\lieg$ (and similarly $U(\liea_{\C})$,
$U(\lien_{\C})$ \ldots{}). In such an algebra we let $U(\lieg_{\C})^{\leq d}$
denote the subspace generated by all (non-commutative) monomials in
$\lieg_{\C}$ of degree at most $d$.

\subsection{Representation Theory}

For any continuous representation of $K$ on a Fréchet space $W$,
and $\tau\in\hat{K}$ we let $W_{\tau}$ denote the $\tau$-isotypical
subspace, and $W_{K}=\oplus_{\tau}W_{\tau}$ denote the (dense) subspace
of $K$-finite vectors. We let $\hat{W}_{K}=\prod_{\tau}W_{\tau}$
denote the completion of $W_{K}$ with respect to this decomposition.
This is the space of formal sums $\sum_{\tau}w_{\tau}$ where $w_{\tau}\in W_{\tau}$.
We endow $\hat{W}_{K}$ with the product topology, which is also the
topology of convergence component-wise.

We specifically set $V=\Vc$ with the right regular action of $K$
and let $\VK$ denote the space of $K$-finite vectors there. We also
have $\VK={L^{2}(\mk)}_{K}$; $\VKh$ can be identified with the algebraic
dual $\VK'$ via the pairing $\left(f,\sum_{\tau}\phi_{\tau}\right)\mapsto\sum_{\tau}\int_{\mk}f\cdot\phi_{\tau}$;
the product topology is the weak-{*} topology. We let $\varphi_{0}\in\VK$
denote the function everywhere equal to $1$.
\begin{defn}
For $\nu\in\acs$ let $G$ act by the right regular representation
on
\[
\mathcal{F}^{\nu}=\left\{ \varphi\in C^{\infty}(G)\vert\varphi(namg)=a^{\nu+\rho}\varphi(g)\right\} .
\]

This induces a $\gK$-module structure on the space of $K$-finite
vectors $\mathcal{F}_{K}^{\nu}$. By the Iwasawa decomposition the
restriction map $\mathcal{F}_{K}^{\nu}\to\VK$ is an isomorphism of
algebraic representations of $K$, giving us a model $(I_{\nu},V_{K})$
for $F_{K}^{\nu}$. Given $\Phi=\sum_{\tau}\phi_{\tau}\in\VKh$ and
$X\in\lieg$ we set $I_{\nu}(X)\Phi=\sum_{\tau}I_{\nu}(X)\phi_{\tau}$
(the $\tau'$-component of the sum only has contribution from $K$-types
appearing in the tensor product of $\tau'$ and the adjoint representation
of $K$ on $\lieg$). Let $\bar{\mathds{1}}$ denote the trivial representation
of $\gK$ where the complex number $z$ acts by multiplication by
$\bar{z}$. Let $\left(\bar{I}_{\nu},\VKh\right)=\left(I_{\nu},\VKh\right)\otimes\bar{\mathds{1}}$.\end{defn}
\begin{notation}
Let $\left(\mathcal{I}_{\nu},\mathcal{V}_{K}\right)$ denote the $\gK$
module $\left(I_{\nu}\otimes\bar{I}_{\nu},\VK\otimes\VKh\right)$.\end{notation}
\begin{fact}
(Induced Representations)
\begin{enumerate}
\item The pairing $\left(f,g\right)\mapsto\int_{\mk}fg$ is a $G$-invariant
pairing on $\mathcal{F}^{\nu}\otimes\mathcal{F}^{-\nu}$. Equivalently,
$\left(f,g\right)\mapsto\int_{M\backslash K}f\bar{g}$ is an invariant
Hermitian pairing between $(I_{\nu},V_{K})$ and $(I_{-\bar{\nu}},V_{K})$.
For $\nu\in i\ars$ (the \emph{unitary axis}) it follows that $(I_{\nu},V_{K})$
is unitarizable, its invariant Hermitian form given by the standard
pairing of $L^{2}(\mk)$.
\item The induced representation is irreducible for $\nu$ lying in an open
dense subset of $i\ars$.
\item Every irreducible spherical $(\mathfrak{g},K)$-module $(\pi,V_{\pi})$
can be realized as a quotient via an intertwining operator $R\colon(I_{\nu},V_{K})\to(\pi,V_{\pi})$,
for some $\nu\in\cC$.
\end{enumerate}
\end{fact}

\subsection{Intertwining Operators}

Given $w\in W$ and $\nu\in\acs$, we can uniquely extend any $\varphi\in V_{K}$
to an element of $\cF^{\nu}$ (also denoted $\varphi)$. For $\nu\in\oCw$
we can then define an endomorphism $A(\nu;w)$ of $V_{K}$ by 
\[
\left(A(\nu;w)\varphi\right)(k)=\int_{\bar{N}\cap wNw^{-1}}\varphi(\bar{n}wk)d\bar{n}
\]
(the integral converges absolutely in this case). It is easy to check
that this operator intertwines the representations $(I_{\nu},V_{K})$
and $(I_{w\nu},V_{K})$ and is holomorphic in the domain $\Omega_{w}$.
\begin{fact}
\label{fac: Int-ops}(Intertwining operators)\end{fact}
\begin{enumerate}
\item \cite[Prop.\ 60(i)]{KnappStein:IntOps} The operators $A(\nu;w)$
admit a meromorphic continuation to all of $\mathfrak{a}_{\C}^{*}$,
intertwining the representations $(I_{\nu},V_{K})$ and $(I_{w\nu},V_{K})$.
For $\nu\in i\ars$ they are unitary operators.
\item \cite[{\S}VII.5]{Knapp:Rpn_Th_SS_Gps} For $\nu\in\mathfrak{a}_{\C}^{*}$
and $\beta\in\Delta$ set $\nu_{\beta}=\frac{2\left\langle \nu,\beta\right\rangle }{\left\langle \beta,\beta\right\rangle }$.
For $w\in W$ set $\Phi_{w}=\left\{ \beta\in\Delta\setminus2\Delta\vert\beta\in\Delta^{+}\cap w^{-1}\Delta^{-}\right\} $.
Then $A(\nu;w)\varphi_{0}=r(\nu;w)\varphi_{0}$ where 
\[
r(\nu;w)=\prod_{\beta\in\Phi_{w}}\left[\frac{\Gamma(p_{\beta}+q_{\beta})}{\Gamma(\frac{1}{2}(p_{\beta}+q_{\beta}))}\frac{\Gamma(\frac{1}{2}\nu_{\beta})}{\Gamma(\frac{1}{2}(\nu_{\beta}+p_{\beta}))}\frac{\Gamma(\frac{1}{4}(\nu_{\beta}+p_{\beta}))}{\Gamma(\frac{1}{4}(\nu_{\beta}+p_{\beta})+\frac{1}{2}q_{\beta})}\right].
\]
We set $\tilde{A}(\nu;w)=r^{-1}(\nu;w)A(\nu;w)$.
\item \cite[Ch.\ XVI]{Knapp:Rpn_Th_SS_Gps}If the spherical representation
$(\pi,V_{\pi})$ is unitarizable and realized as a quotient of $(I_{\nu},V_{K})$
as before, these exists $w\in W$ with $w^{2}=1$ such that $w\nu=-\bar{\nu}$;
further more $\Re(\nu)$ belongs to a fixed compact set.
\item Conversely, let $w\in W$ satisfy $w^{2}=1$, and let $\nu\in\acs$
such that $w\nu=-\bar{\nu}$. Then 
\[
\left(f,g\right)\mapsto\left\langle A(\nu;w)f,g\right\rangle _{L^{2}(K)}
\]
defines a non-zero $(\mathfrak{g},K)$-equivariant Hermitian pairing
on $(I_{\nu},V_{K})$; the subspace where the pairing vanishes is
the kernel of $A(\nu;w)$ and the quotient is irreducible. The quotient
is unitarizable iff the pairing is semidefinite, and every unitary
spherical representation arises this way.
\item \cite{Arthur:IntOpsI} For fixed $\varphi,\psi\in V_{K}$ the matrix
coefficient 
\[
\nu\mapsto\left\langle \tilde{A}(\nu;w)\varphi,\psi\right\rangle _{L^{2}(K)}
\]
 is a rational function of $\nu$ where we identify $\acs$ with $\C^{\dim\liea}$
via the map $\nu\mapsto(\nu(H_{\alpha}))_{\alpha\in\Pi}$.\end{enumerate}
\begin{rem}
Since $V_{K}$ contains a unique copy of the trivial representation
of $K$, we must have $A(\nu;w)\varphi_{0}=r(\nu;w)\varphi_{0}$ for
some meromorphic function $r(\nu;w)$. Showing the integral defining
$r(\nu;w)$ converges absolutely for $\nu\in\oCw$ proves the absolute
convergence claim above.

Since $r(\nu;w)$ does not vanish in open domain $\Omega_{w}$, $\tilde{A}(\nu;w)$
cannot have zeroes or poles there.
\end{rem}

\section{\label{sec: Int-bds}Interpolation bounds on intertwining operators}
\begin{lem}
Let $f\in\C(z)$ be a rational function of one variable. Assume that
$f$ is bounded on the line $\Re(z)=0$ and has no poles to the right
of the line. Then 
\[
\sup\left\{ f(z)\mid\Re(z)\geq0\right\} =\sup\left\{ f(z)\mid\Re(z)=0\right\} .
\]
\end{lem}
\begin{proof}
Composing with a Möbius transformation we may instead consider the
case of a rational function $f$ holomorphic in the interior of the
unit disk $\mathbb{D}$ and bounded on $\partial\mathbb{D}\setminus\left\{ 1\right\} $.
The singularity of $f$ at $z=1$ is at most a pole since $f$ is
rational. The boundedness on the rest of the boundary then shows the
singularity is removable so that $f$ is continuous on the closed
disk. Finally, apply the usual maximum principle.\end{proof}
\begin{thm}
Let $w\in W$, and let $\tilde{A}(\nu;w)\colon\left(I_{\nu},V_{K}\right)\to\left(I_{w\nu},V_{K}\right)$
be the intertwining operator, normalized such that $\tilde{A}(\nu;w)\varphi_{0}=\varphi_{0}$.
Then $\left\Vert \tilde{A}(\nu;w)\right\Vert _{L^{2}(K)}\leq1$ for
$\nu\in\bar{\Omega}_{w}$.\end{thm}
\begin{proof}
By duality, it suffices to show that 
\[
\left\langle \tilde{A}(\nu;w)\varphi,\psi\right\rangle _{L^{2}(K)}\leq\left\Vert \varphi\right\Vert _{L^{2}(K)}\left\Vert \psi\right\Vert _{L^{2}(K)}
\]
holds for all non-zero $\varphi,\psi\in V_{K}$ and all $\nu$ as
above. As the left-hand-side is a meromorphic function of $\nu$,
it suffices to establish the inequality for $\Re(\nu)\in\oCw$, which
we assume henceforth.

We restrict the left-hand-side to a one-parameter family of spectral
parameters by considering the meromorphic one-variable function 
\[
f(z)=\frac{1}{\left\Vert \varphi\right\Vert _{L^{2}(K)}\left\Vert \psi\right\Vert _{L^{2}(K)}}\left\langle \tilde{A}(i\Im(\nu)+z\Re(\nu);w)\varphi,\psi\right\rangle _{L^{2}(K)}.
\]
It will be convenient to write $\nu_{z}=i\Im(\nu)+z\Re(\nu)$ so that
$\nu_{1}=\nu$, and note that the parameters in our family satisfy
$\Re(\nu_{z})=\Re(z)\Re(\nu)$ and in particular $\Re(\nu_{z})\in\oCw$
when $\Re(z)>0$. Arthur's result quoted above (Fact \ref{fac: Int-ops}(4))
is that $f(z)$ is a rational function of $z$. It has no poles in
the domain $\Re(z)>0$ since the intertwining operator has no poles
in $\Omega_{w}$. When $z=it\in i\R$, the parameter $\nu_{z}\in i\ars$
is unitary and hence $\tilde{A}(\nu_{z};w)$ is a unitary operator,
which implies $\left|f(z)\right|\leq1$ by Cauchy-Schwartz. In particular,
$f$ has no poles on this line, and the claim now follows from the
Lemma.
\end{proof}

\paragraph*{Proof of Theorem \ref{mainthm: intops}}

Let $\left(\pi,V_{\pi}\right)\in\hat{G}$ be spherical, and let $R\colon\left(I_{\nu},V_{K}\right)\to\left(\pi,V_{\pi}\right)$
be a non-zero intertwining operator with the real part of $\nu\in\acs$
in the closed positive chamber $\cC$, normalized such that $\left\Vert R(\varphi_{0})\right\Vert _{V_{\pi}}=1$.

By Fact \ref{fac: Int-ops}(3) there exists an involution $w\in W$
such that $w\nu=-\bar{\nu}$ and such that $\left\langle \varphi,A(\nu;w)\psi\right\rangle _{L^{2}(K)}$
is a $G$-equivariant Hermitian pairing on $\left(\mathcal{I}_{\nu},V_{K}\right)$.
Also, the image of $A(w,\nu)$ is irreducible (in fact, isomorphic
to $\pi$). By Schur's Lemma there is $c\geq0$ such that for all
$K$-finite $\varphi$ we have $\left\Vert R(\varphi)\right\Vert _{V_{\pi}}=c\left\langle \tilde{A}(\nu;w)\varphi,\varphi\right\rangle _{L^{2}(K)}$
. Our normalization implies that the constant of proportionality is
$1$, and the bound on the intertwining operator gives the claim $\left\Vert R(\varphi)\right\Vert _{V_{\pi}}\leq\left\Vert \varphi\right\Vert _{L^{2}(K)}$.

\subsection{\label{sub:SL2-example}Example: $\SL_{2}(\R)$}

Let $G=\SL_{2}(\R)$, $K=\SO_{2}(\R)$. The Lie algebra $\lieg=\sl_{2}(\R)$
is spanned by the three elements $H=\left(\begin{array}{cc}
1\\
 & -1
\end{array}\right)$, $X=\left(\begin{array}{cc}
0 & 1\\
 & 0
\end{array}\right)$ and $\bar{X}=\left(\begin{array}{cc}
0\\
1 & 0
\end{array}\right)$. Individually they span the subalgebras $\liea=\R H$, $\lien=\R X$
and $\lienb=\R\bar{X}$. These are the Lie algebras of the subgroups
$A=\left\{ \left(\begin{array}{cc}
*\\
 & *
\end{array}\right)\right\} $, $N=\left\{ \left(\begin{array}{cc}
1 & x\\
 & 1
\end{array}\right)\right\} $, $\bar{N}=\left\{ \left(\begin{array}{cc}
1\\
* & 1
\end{array}\right)\right\} $. We shall also use $M=Z_{K}(A)=\left\{ \pm I\right\} $ and fix $w=\left(\begin{array}{cc}
 & 1\\
-1
\end{array}\right)$, a representative for the non-trivial class in $W(\lieg:\liea)\isom N_{K}(A)/Z_{K}(A)$.
Letting $k_{\phi}=\left(\begin{array}{cc}
\cos\phi & \sin\phi\\
-\sin\phi & \cos\phi
\end{array}\right)$ so that $K=\left\{ k_{\phi}\right\} _{\phi\in\R/2\pi\Z}$, we normalize
the Haar measures on the circles $K$ and $\mk$ to be probability
measures, on $\bar{N}$ to be $\frac{1}{\pi}du$ where $\bar{n}(u)=\exp(u\bar{X})$.

As $\left[H,X\right]=2X$, we have $[tH,X]=\alpha(tH)X$ for that
$\alpha\in\mathfrak{a}_{\R}^{*}$ (the {}``positive root'') given
by $\alpha(tH)=2t$. We then set $\rho(tH)=\frac{1}{2}\alpha(tH)=t$
({}``half the sum of the positive roots''). We can then identify
the complex dual $\mathfrak{a}_{\C}^{*}$ with $\C$ via $z\mapsto(tH\mapsto zt)$.

The induced representation $\mathscr{P}^{+,z}$ (cf \cite[\S\S2.5 \& 7.1]{Knapp:Rpn_Th_SS_Gps})
is the right regular representation of $G$ on the space
\[
\mathcal{F}^{+,z}=\left\{ F\in C^{\infty}(G)\mid F(n\exp(tH)mg)=e^{(z+1)t}F(g)\right\} .
\]
By the Iwasawa decomposition these functions are uniquely determined
by their restriction to the space $V=C(\mk)$. the space of even functions
on the circle. As usual we shall restrict our attention to the subspace
$V_{K}\subset V$ of even trigonometric polynomials (the {}``$K$-finite''
vectors), which is spanned by the Fourier modes $\varphi_{2m}(\theta)=\exp(2mi\theta)$.

As we will see shortly, for $\Re(z)>0$ and $F\in\mathcal{F}^{+,z}$
the integral $(AF)(g)=\int_{\bar{N}}F(\bar{n}wg)d\bar{n}$ converges
absolutely. Assuming this, we now verify that it defines an element
of $\mathcal{F}^{+,-z}$. It also clearly intertwines the right regular
representations under consideration.

For $a=\exp(tH)\in A$ we note that $waw^{-1}=a^{-1}$ and that for
$\bar{n}'=a\bar{n}a^{-1}\in\bar{N}$ we have $d\bar{n}'=e^{-2t}d\bar{n}$.
From this we conclude:

\[
AF(ag)=\int_{\bar{N}}F(\bar{n}wag)d\bar{n}=e^{(-z+1)t}\int_{\bar{N}}F(\bar{n}'wg)d\bar{n}'.
\]
Similarly we note that for $n\in N$, $wnw^{-1}\in\bar{N}$. Since
we are integrating \wrt to a Haar measure on $\bar{N}$, this shows
that $AF(ng)=AF(g)$. Finally, since $M$ is central it is clear that
$AF(mg)=AF(g)$. The smoothness is clear by differentiating under
the integral sign, and the operator preserves $K$-finiteness since
it commutes with the action of $K$.

Given $u\in\R$ we set $t=-\frac{1}{2}\log(1+u^{2})$ and define $\theta\in(0,\pi)$
by $u=\cot\theta$. Then there exists $n\in\R$ such that:
\[
\bar{n}(u)w=\left(\begin{array}{cc}
1 & n\\
 & 1
\end{array}\right)\exp(tH)k_{\theta}.
\]
Since $du=-\frac{d\theta}{\sin^{2}\theta}$ and $e^{t}=\left|\sin\theta\right|$
we get:
\begin{equation}
A(z)F(k_{\phi})=\frac{1}{\pi}\int_{0}^{\pi}\left|\sin\theta\right|^{z-1}F(k_{\phi+\theta})d\theta.\label{eq:intop-K-sl2}
\end{equation}
Since $F\restrict_{K}$ is even this is a convolution operator on
the circle $K$.

We can now address the question of convergence. Taking absolute values
and bounding $F(k_{\phi+\vartheta})$ by $\left\Vert F\right\Vert _{L^{\infty}(K)}$
it is clear that $A(z)$ converges absolutely for all $F\in C(\mk)$
iff the same holds for $A(\Re(z))\varphi_{0}$ where $\varphi_{0}$
is the constant function. When $F\restrict_{K}$ is one of the Fourier
modes $\varphi_{2m}$, we find on page 8 of \cite{MagnusOberhettingerSoni:SpecialFunc}
that the integral (\ref{eq:intop-K-sl2}) converges absolutely for
$\Re(z)>0$ (the {}``open positive Weyl chamber'') and takes the
value: 
\[
A\varphi_{2m}=(-1)^{m}\frac{2^{1-z}\Gamma(z)}{\Gamma(\frac{z+1}{2}+m)\Gamma(\frac{z+1}{2}-m)}\cdot\varphi_{2m}.
\]

We may thus extend $A(z)$ to a family of operators $A(z)\colon V_{K}\to V_{K}$
intertwining the induced representations and defined everywhere except
for the pole at $z=0$. We next normalize these operators. As above
we define $r(z)$ by $A(z)\varphi_{0}=r(z)\varphi_{0}$, that is:
\[
r(z)=\frac{2^{1-z}\Gamma(z)}{\Gamma(\frac{z+1}{2})^{2}}.
\]
Note that this meromorphic function has no zeroes or poles for $\Re(z)>0$.
In particular, if we set $\tilde{A}(z)F=r(z)^{-1}A(z)F$ the new operator
is also regular for $\Re(z)>0$ and extends meromorphically to $\C$.
It will now have poles for $\Re(z)<0$, but will be regular for $\Re(z)=0$.

The claim of Theorem \ref{mainthm: intops} (in this case) is that
$\tilde{A}(z)\colon V_{K}\to V_{K}$ is bounded in the $L^{2}$ norm.
Since it is diagonal in the Fourier basis it suffices to verify that
the Fourier coefficients $\tilde{A}(z)\varphi_{2m}=c_{2m}(z)\varphi_{2m}$
satisfy $\left|c_{2m}(z)\right|\leq1$ when $\Re(z)\geq0$.

For $m=0$ this is true by definition of $r(z)$. In general, using
$\Gamma(z+m)=\Gamma(z)\prod_{j=0}^{m-1}(z+j)$ we get:
\[
c=\frac{\Gamma(\frac{z+1}{2})^{2}}{\Gamma(\frac{z+1}{2}+m)\Gamma(\frac{z+1}{2}-m)}=\prod_{j=0}^{m-1}\frac{z-(2j+1)}{z+(2j+1)}.
\]
Now $z-(2j+1)$ and $z+(2j+1)$ always have the same imaginary part,
but for $\Re(z)\geq0$ the denominator always has a larger real part
(in absolute value), and the product has magnitude at most $1$ as
claimed.
\begin{rem}
The rationality of the matrix coefficient $c_{2m}(z)=\left(\tilde{A}(z)\varphi_{2m},\varphi_{2m}\right)_{L^{2}(\mk)}$
was an essential ingredient in our argument above.\end{rem}
\begin{cor}
The normalized operator has no poles (or zeroes) for $\Re(\nu)$ in
the \emph{closed} positive chamber.
\end{cor}

\section{\label{sec: do-lift}Degenerate lift}

In this section we establish Theorem \ref{mainthm:lift}.

\subsection{\label{sub: lift-one}The basic construction}

\paragraph*{One eigenfunction}

Let $\psi\in L^{2}(Y)$ be a normalized eigenfunction with the parameter
$\nu\in\bar{\Omega}$; let $R\colon\left(\mathcal{I}_{\nu},V_{K}\right)\to\left(\mathcal{R},L^{2}(X)_{K}\right)$
be an intertwining operator with $R(\varphi_{0})=\psi$. Given $f_{1},f_{2}\in V_{K}$
and $g\in\Test$ we set:
\[
\mu_{R}\left(f_{1},f_{2}\right)(g)\eqdef\int_{X}R(f_{1})\overline{R(f_{2})}g\, d\vol_{X}\,.
\]
By the Cauchy-Schwartz inequality and Theorem \ref{mainthm: intops},
\[
\left|\mu_{n}\left(f_{1},f_{2}\right)(g)\right|\leq\norm{f_{1}}_{L^{2}(K)}\norm{f_{2}}_{L^{2}(K)}\norm{g}_{L^{\infty}(X)}.
\]
In particular, the $\mu_{n}\left(f_{1},f_{2}\right)$ extend to finite
Borel measures on $X$ (positive measures when $f_{1}=f_{2}$). Also,
we have a bound on the total variation of these measures which depends
only on $f_{1}$ and $f_{2}$ but not on $\nu$ or $R$.

This construction extends to the case where one of the two test vectors
is not $K$-finite. Given $\Phi=\sum_{\tau\in\hat{K}}\phi_{\tau}\in\hat{V}_{K}$
we set: 
\[
\mu_{R}\left(f,\Phi\right)(g)=\sum_{\tau\in\hat{K}}\mu_{R}\left(f,\phi_{\tau}\right)(g)\,,
\]
noting that only finitely many $\tau$ can contribute. Letting $\Distr$
denote the algebraic dual of $\Test$, we have obtained a map : 
\[
\mu_{R}\colon V_{K}\times\hat{V}_{K}\to\Distr
\]
which is linear in the first variable and conjugate-linear in the
second. Integration by parts on $\Gamma\backslash G$ shows that the
extension $\mu_{R}\colon\left(\mathcal{I}_{\nu},\mathcal{V}_{K}\right)\to\Distr$
is an intertwining operator for the $\gK$ module structures.
\begin{rem}
By $\Distr$ we mean the \emph{algebraic} dual of our space $\Test$
of test functions. By abuse of terminology we shall call its elements
\emph{distributions}; convergence of distributions will be in the
weak-{*} (pointwise) sense. Apart from limits of uniformly bounded
sequences of measures, the limits we shall consider will be \emph{positive}
distributions (that is, take non-negative values at non-negative test
functions), and such distributions are always Borel measures (finiteness
will require an easy separate argument). For completeness we note,
however, that when $\Phi$ defines a distribution on $\mk$ in the
ordinary sense (as is the case with $\delta$), $\mu_{R}(f,\Phi)$
is bounded \wrt to an appropriate Sobolev norm and hence $\mu_{R}\left(f,\Phi\right)$
is a distribution on $X$ in the ordinary sense. Moreover, the bound
depends on $f$ and on the dual Sobolev norm of $\Phi$ but not $\nu$
or $R$.
\end{rem}

\paragraph*{A sequence of eigenfunctions}

Let $\sequ{\nu}\subset\bar{\Omega}$ such that $\norm{\nu_{n}}\to\infty$,
and let $R_{n}\colon\left(\mathcal{I}_{\nu_{n}},V_{K}\right)\to\left(\mathcal{R},L^{2}(X)_{K}\right)$
be intertwining operators with $\left\Vert R_{n}(\varphi_{0})\right\Vert _{L^{2}(X)}=1$.
Assume that $\mbn=\mu_{n}\left(\varphi_{0},\varphi_{0}\right)$ converge
weak-{*} to a limiting measure $\mbi$, which we would like to study.

Fixing $f_{1},f_{2}\in V_{K}$ the construction of the previous section
gives a sequence of Borel measures $\mu_{n}\left(f_{1},f_{2}\right)=\mu_{R_{n}}\left(f_{1},f_{2}\right)$
all of which have total variation at most $\norm{f_{1}}_{L^{2}(K)}\norm{f_{2}}_{L^{2}(K)}$.
By the Banach-Alaoglu theorem there exists a subsequence $\left\{ n_{k}\right\} _{k=1}^{\infty}\subset\N$
such that $\mu_{n_{k}}\left(f_{1},f_{2}\right)$ converge weak-{*}.
Fixing a countable basis $\left\{ \varphi_{i}\right\} _{i=1}^{\infty}\subset V_{K}$,
by the standard diagonalization argument we may assume (after passing
to a subsequence) that for any $f_{1},f_{2}\in V_{K}$ there exists
a measure $\mi\left(f_{1},f_{2}\right)$ such that for all $g\in\Test$,
\[
\lim_{n\to\infty}\mn\left(f_{1},f_{2}\right)(g)=\mi\left(f_{1},f_{2}\right)(g).
\]
As before, given $f_{1}$ and $g$, the value of $\mi\left(f_{1},f_{2}\right)(g)$
only depends on the projection of $f_{2}$ to a finite set of $K$-types.
We can thus extend $\mi$ to all of $\mathcal{V}_{K}=V_{K}\otimes\hat{V}_{K}$
and it is clear that $\mn$ converge weak-{*} to $\mi$ in the sense
that for any fixed $F\in\mathcal{V_{K}}$ and $g\in\Test$, $\lim_{n\to\infty}\mn(F)(g)=\mi(F)(g)$.

The asymptotic properties of $\mn$ are governed by the normalized
spectral parameters $\ntn=\frac{\nu_{n}}{\left\Vert \nu_{n}\right\Vert }$;
passing to a subsequence again we assume $\ntn\to\nti$ as $n\to\infty$.
Since the $\Re(\nu_{n})$ are uniformly bounded (we are dealing with
unitary representations), the limit parameter $\nti$ is purely imaginary.
\begin{defn}
Call the sequence of intertwining operators $\left\{ R_{n}\right\} _{n=1}^{\infty}$
\emph{conveniently arranged} if $\ntn$ converge to some $\nti\in i\ars$
and if for any $f_{1},f_{2}\in V_{K}$ the sequence of measures $\left\{ \mn\left(f_{1},f_{2}\right)\right\} _{n=1}^{\infty}$
converges in the weak-{*} topology.
\end{defn}
Given our limiting measure $\mbi$ we now fix once and for all a conveniently
arranged sequence $R_{n}$ such that $\mn\left(\varphi_{0},\varphi_{0}\right)$
converges to $\mbi$, and set $M_{1}=Z_{K}(\nti)$. The motivation
for the following choice will be come clear in the following Section.
\begin{defn}
\label{def: Lift}Let $\delta_{1}\in V_{K}'$ be the distribution
$\delta_{1}(f)=\int_{M\backslash M_{1}}f(m_{1})dm_{1}$. Set: 
\[
\mn=\mn\left(\varphi_{0}\otimes\delta_{1}\right)\,,
\]
which converge to the limit $\mi=\mi\left(\varphi_{0}\otimes\delta_{1}\right)$.
\end{defn}
Note that for a $K$-invariant test function $g$, $\mn(g)=\mbn(g)$
since the spherical part of $\delta_{1}$ is exactly $\varphi_{0}$.
It follows that the $\mn$ indeed are lifts of the measures $\mbn$
to $X=\Gamma\backslash G$, which is Claim \eqref{enu: m-Lift} of
the main Theorem.
\begin{rem}
Note that our definition of $\mn$ (and hence $\mi$) depends on the
limit point $\nti$, and not only on the limiting measure $\mbi$.
\end{rem}

\subsection{Integration by parts; positivity.}

Pointwise addition and multiplication give an algebra structure to
$\VK$. Our asymptotic calculus for the measures $\mn\left(f_{1},f_{2}\right)$
will depend on the the following elements of this algebra.

For $X\in\lieg$ and $k\in K$ we write the Iwasawa decomposition
of $\Ad(k)X$ as $X_{\lien}(k)+X_{\liea}(k)+X_{\liek}(k)$. Now for
$X\in\lieg$ and $\nti\in i\ars$ set: 
\[
p_{X}(k)=\frac{1}{i}\left\langle X_{\liea}(k),\nti\right\rangle 
\]
this is a left-$M_{1}$-invariant function on $K$, in particular
a left-$M$-invariant function on $K$. It is $K$-finite, being a
matrix element of the adjoint representation of $K$ on $\lieg$.
\begin{lem}
The subalgebra of $V_{K}$ generated by $\left\{ \varphi_{0}\right\} \cup\left\{ p_{X}\right\} _{X\in\lieg}$
under pointwise addition and multiplication \textup{is precisely}
$\mathcal{F}_{1}=C(M_{1}\backslash K)_{K}$, the algebra of left-$M_{1}$
invariant, right $K$-finite functions on $K$.\end{lem}
\begin{proof}
This follows from the Stone-Weierstrass Theorem, by which it suffices
to check that the functions $p_{X}$ separate the points of $M_{1}\backslash K$.
Indeed, if $p_{X}(k)=p_{X}(k')$ for all $X$ then $M_{1}k'=M_{1}k$
-- recall that $M_{1}$ was defined as the centralizer of $\nti$.
\end{proof}
Our calculation depends on the following basic formula, obtained by
integration by parts:
\begin{lem}
(\cite[Lem.\ 3.10 \& Cor.\ 3.11]{SilbermanVenkatesh:SQUE_Lift}) There
exists a norm $\left\Vert \cdot\right\Vert $ on $\Cic(X)_{K}$ such
that for any $f_{1},f_{2}\in V_{K}$ and $X\in\lieg$, 
\[
\left|\mn\left(p_{X}f_{1},f_{2}\right)(g)-\mn\left(f_{1},\overline{p_{X}}f_{2}\right)(g)\right|\ll_{f_{1},f_{2}}\left\Vert g\right\Vert \left[\left\Vert \ntn-\nti\right\Vert +\left\Vert \nu_{n}\right\Vert ^{-1}\right]\,.
\]
\end{lem}
\begin{cor}
\label{cor: int-parts}Let $f\in\mathcal{F}_{1}$ and $f_{1},f_{2}\in V_{K}$
. Then, for any $g\in\Cic(X)_{K}$, 
\[
\left|\mu_{n}\left(f\cdot f_{1},f_{2}\right)(g)-\mu_{n}\left(f_{1},\overline{f}\cdot f_{2}\right)(g)\right|\ll_{f,f_{1},f_{2}}\left\Vert g\right\Vert \left[\left\Vert \ntn-\nti\right\Vert +\left\Vert \nu_{n}\right\Vert ^{-1}\right]\,.
\]

\end{cor}
Claims \eqref{enu: m-positive} and \eqref{enu: m-equiv} of the main
Theorem now follow from:
\begin{prop}
We can choose $f_{n}\in V_{K}$ (in the notation of the main Theorem,
set $\tilde{\psi}_{n}=R_{n}(f_{n})$) so that the measures $\sigma_{n}=\mu_{n}(f_{n},f_{n})$
converge weak-{*} to $\mu_{\infty}$.\end{prop}
\begin{proof}
Let $\left\{ h_{k}\right\} _{k=1}^{\infty}\in\cF_{1}$ be real-valued
functions such that $h_{k}^{2}$ converge weak-{*} to $\delta_{1}$,
and let $h_{0}=\varphi_{0}$ (it is easy to see that such a sequence
exists). By Corollary \ref{cor: int-parts} there exists constants
$C_{k}$ depending only on the choice of $f_{k}$ such that for any
$g\in\Test$ and $n$, 
\[
\left|\mn\left(\varphi_{0},h_{k}^{2}\right)(g)-\mn\left(h_{k},h_{k}\right)(g)\right|\leq C_{k}\left\Vert g\right\Vert \left[\left\Vert \ntn-\nti\right\Vert +\left\Vert \nu_{n}\right\Vert ^{-1}\right]\,.
\]
Noting that $C_{0}=0$, given $n\geq1$ let $k(n)$ be the maximal
$k\in\left\{ 0,\cdots,n\right\} $ such that $C_{k}\leq\left[\left\Vert \ntn-\nti\right\Vert +\left\Vert \nu_{n}\right\Vert ^{-1}\right]^{-1/2}$,
and set $f_{n}=h_{k(n)}$, $\sn=\mn(f_{n},f_{n})$. The sequence $k(n)$
is monotone and tends to infinity; it follows that $f_{n}^{2}$ converge
weakly to $\delta_{1}$.

Finally, we have: 
\[
\left|\mn(g)-\sn(g)\right|\leq\left|\mn\left(\varphi_{0},\delta_{1}-f_{n}^{2}\right)(g)\right|+\left[\left\Vert \ntn-\nti\right\Vert +\left\Vert \nu_{n}\right\Vert ^{-1}\right]^{1/2}\left\Vert g\right\Vert \,.
\]
Let $T\subset\hat{K}$ be a finite subset such that $g\in\sum_{\tau\in T}\Cic(X)_{\tau}$.
Let $d_{n}\in\sum_{\tau\in T}V_{\tau}$ be the projection of $\delta_{1}-f_{n}^{2}$
to that space. Then $\overline{R(\delta_{1}-f_{n}^{2}-d_{n})}$ has
trivial pairing with $R(\varphi_{0})g$, since they don't transform
under the same $K$-types. We may thus bound the first term in the
inequality above by $\left|\mn\left(\varphi_{0},d_{n}\right)(g)\right|\leq\left\Vert d_{n}\right\Vert _{L^{2}(K)}\left\Vert g\right\Vert _{L^{\infty}(X)}$.
Since $\sum_{\tau\in T}V_{\tau}$ is finite-dimensional, that $d_{n}\to0$
weakly implies that $d_{n}\to0$ in norm. Since $\mn(g)\to\mi(g)$
we conclude that $\sn(g)\to\mi(g)$ as well.\end{proof}
\begin{cor}
\label{cor: positive}$\mi$ extends to a non-negative measure on
$X$ of total mass at most $1$. When $X$ is compact $\mi$ is a
probability measure. \end{cor}
\begin{proof}
The $\sigma_{n}$ extend to positive measures, hence $\mi$ extends
to a non-negative measure. To bound the total mass it suffices to
consider $K$-invariant test functions for which $\mi$ agrees with
$\mbi$, a weak-{*} limit of probability measures.\end{proof}
\begin{cor}
\label{cor: equiv}When $\psi_{n}$ are eigenfunctions of an algebra
$\mathcal{H}$ of operators which commute with the $G$-action, then
so are $\tilde{\psi}_{n}$.\end{cor}
\begin{proof}
By Schur's Lemma each element of $\mathcal{H}$ acts as a scalar on
the irreducible representation generated by $\psi_{n}$; $\tilde{\psi}_{n}$
belongs to this representation.
\end{proof}

\subsection{\label{sub: A-inv}$A_{1}$-invariance.}

Let $\delta\in V_{K}'\isom\hat{V}_{K}$ be the delta distribution,
that is $\delta(f)=f(1)$. Since $\liea$ is a quotient of $\liem\oplus\liea\oplus\lien$
by a Lie ideal, we can consider any $\lambda\in\acs$ as a Lie algebra
homomorphism $\liea_{\C}\oplus\liem_{\C}\oplus\lien_{\C}\to\C$. It
thus extends to an algebra homomorphism $U(\liea_{\C}\oplus\liem_{\C}\oplus\lien_{\C})\to\C$,
and there exists a unique algebra endomorphism $\tau_{\lambda}\colon U(\liea_{\C}\oplus\liem_{\C}\oplus\lien_{\C})\to U(\liea_{\C}\oplus\liem_{\C}\oplus\lien_{\C})$
such that $\tau_{\lambda}(X)=X+\lambda(X)$ for $X\in\liea_{\C}\oplus\liem_{\C}\oplus\lien_{\C}$.
\begin{lem}
Let $\nu\in\acs$, $u\in U(\liea_{\C}\oplus\liem_{\C}\oplus\lien_{\C})$
\begin{enumerate}
\item $I_{\nu}(u)\delta=(-\rho+\nu)(u)\cdot\delta$.
\item $\mathcal{I}_{\nu}(\tau_{\rho+\nu-2\Re(\nu)}(u))\left(f\otimes\delta\right)=\left(I_{\nu}(u)f\right)\otimes\delta$.
\end{enumerate}
\end{lem}
\begin{proof}
By induction it suffices to prove both assertions for $u=X\in\liem\oplus\liea\oplus\lien$.
The first claim follows from the the invariant pairing of $\mathcal{F}^{\nu}$
with $\mathcal{F}^{-\nu}$. Taking complex conjugates, this implies:
\[
\mathcal{I}_{\nu}(X)\left(f\otimes\delta\right)=\left(I_{\nu}(X)f\right)\otimes\delta+\left\langle -\rho_{\liea}+\bar{\nu},X\right\rangle (f\otimes\delta)\,,
\]
which is the second assertion.
\end{proof}
We next summarize the analysis of the center of the universal enveloping
algebra done in \cite[\S4]{SilbermanVenkatesh:SQUE_Lift}.
\begin{prop}
Let $\mathcal{P}\in U(\lieh_{\C})^{\CW}$ be homogeneous of degree
$d$. Then there exist elements $b=b(\mathcal{P})\in U(\lien_{\C})U(\liea_{\C})^{\leq d-2}$
and $c=c(\mathcal{P})\in U(\lieg_{\C})\liek_{\C}$ so that 
\[
z=\tau_{-\rho_{\lieh}}(\mathcal{P})+b+c
\]
belongs to the center of the universal enveloping algebra. Furthermore,
$z$ acts on $\left(I_{\nu},V_{K}\right)$ with the eigenvalue $\mathcal{P}(\nu+\rho_{\liea}-\rho_{\lieh})$.
\end{prop}
It follows that for such $\mathcal{P}$ we have:
\[
\mathcal{I}_{\nu}\left(\tau_{\rho_{\liea}-\rho_{\lieh}+\nu-2\Re(\nu)}(\mathcal{P})+\tau_{\rho_{\liea}+\nu-2\Re(\nu)}(b)-\mathcal{P}(\nu+\rho_{\liea}-\rho_{\lieh})\right)\left(\varphi_{0}\otimes\delta\right)=0
\]
(to see this unwind the definitions, using the fact that $I_{\nu}(c)\varphi_{0}=0$).

Thinking of $\mathcal{P}$ as a function on $\lieh_{\C}^{*}$, let
$P'(\nu)$ denote its differential at $\nu\in\acs$. This is an element
of the cotangent space to $\lieh_{\C}^{*}$, that is an element of
$\lieh_{\C}$.
\begin{prop}
Let $\mathcal{P}\in U(\lieh_{\C})^{\CW}$. Then there exists a polynomial
map $J\colon\acs\to U(\lieg_{\C})$ ($\acs$ thought of as a \emph{real}
vector space), of degree at most $d-2$ in the parameters $\Im(\nu)$,
such that for any unitarizable parameter $\nu\in\acs$, 
\[
\mathcal{I}_{\nu}\left(\mathcal{P}'(\tilde{\nu})+\frac{J(\nu)}{\left\Vert \nu\right\Vert ^{d-1}}\right)\left(\varphi_{0}\otimes\delta\right)=0\,.
\]
\end{prop}
\begin{proof}
Since $\mathcal{P}'$ is a homogeneous polynomial of degree $d-1$,
it suffices to show that $\tau_{\rho_{\liea}-\rho_{\lieh}+\nu-2\Re(\nu)}(\mathcal{P})+\tau_{\rho_{\liea}+\nu-2\Re(\nu)}(b)-\mathcal{P}(\nu+\rho_{\liea}-\rho_{\lieh})-\mathcal{P}'(\nu)$
is a polynomial of degree at most $d-2$ in $\nu$. It is clear that
$J_{1}(\nu)=\tau_{\rho_{\liea}+\nu-2\Re(\nu)}(b)$ is such a polynomial,
as is $J_{2}(\nu)=\tau_{\rho_{\liea}-\rho_{\lieh}+\nu-2\Re(\nu)}(\mathcal{P})-\mathcal{P}(\nu+\rho_{\liea}-\rho_{\lieh}-2\Re(\nu))-\mathcal{P}'(\nu+\rho_{\liea}-\rho_{\lieh}-2\Re(\nu))$.
Since$\mathcal{P}'$ is a polynomial of degree $d-1$ (valued in $\liea_{\C}$),
$J_{3}(\nu)=\mathcal{P}'(\nu+\rho_{\liea}-\rho_{\lieh}-2\Re(\nu))-P'(\nu)$
is also of degree at most $d-2$. It remains to consider $J_{4}(\nu)=\mathcal{P}(\nu+\rho_{\liea}-\rho_{\lieh}-2\Re(\nu))-\mathcal{P}(\nu+\rho_{\liea}-\rho_{\lieh})$
which is a polynomial map of degree $d-1$ in $\nu$.

The first two terms are difference of the values of a polynomial at
two points, we may write this in the form $\left\langle \mathcal{P}'(\nu+\rho_{\liea}-\rho_{\lieh}),-2\Re(\nu)\right\rangle +J_{5}(\nu)$
where $J_{5}(\nu)$ includes the terms of degree $d-2$ or less and
the pairing is the one between $\lieh_{\C}$and $\acs$. Finally,
$\mathcal{P}'(\nu+\rho_{\liea}-\rho_{\lieh})-\mathcal{P}'(\Im(\nu))$
has degree $d-2$ in $\Im(\nu)$. Setting $J_{6}(\nu)=\left\langle \mathcal{P}'(\nu+\rho_{\liea}-\rho_{\lieh})-\mathcal{P}'(\Im(\nu)),-2\Re(\nu)\right\rangle $
we see that $\varphi_{0}\otimes\delta$ is annihilated by:
\[
P'(\nu)+\sum_{i=1}^{6}J_{i}(\nu)-2\left\langle \mathcal{P}'(\Im(\nu)),\Re(\nu)\right\rangle \,.
\]

We conclude by showing that the final scalar vanishes. By assumption
there exists $w\in W$ such that $w\nu=-\bar{\nu}$. By Corollary
\ref{cor: complex} there exist $\tilde{w}\in\CW$ such that $\tilde{w}\nu=-\bar{\nu}$.
Applying the chain rule to $\mathcal{P}=\mathcal{P}\circ\tilde{w}$
we see that $\mathcal{P}'(\Im(\nu))$ is fixed by $\tilde{w}$, while
$\tilde{w}\Re(\nu)=-\tilde{w}\Re(\nu)$.\end{proof}
\begin{cor}
Let $\left\{ R_{n}\right\} _{n=1}^{\infty}$ be a conveniently arranged
sequence of intertwining operators from $(I_{\nu_{n}},V_{K})$ to
$L^{2}(X)$. Then the limit distribution $\mi=\mi\left(\varphi_{0}\otimes\delta_{1}\right)$
is $H$-invariant for any $H=\mathcal{P}'(\nti)$, where $\mathcal{P}\in U(\lieh_{\C})^{\CW}$.\end{cor}
\begin{proof}
By additivity it suffices to prove this when $\mathcal{P}$ is homogeneous.
Next, since $\mn=\mn\left(\varphi_{0}\otimes\delta_{1}\right)$ are
$M_{1}$-invariant distributions (in fact, we are lifting to $\Gamma\backslash G/M_{1}$,
not to $\Gamma\backslash G$), it suffices to consider $M_{1}$-invariant
test functions $g\in\Test^{M_{1}}$. For these we have $\mn\left(\varphi_{0}\otimes\delta_{1}\right)(g)=\mn\left(\varphi_{0}\otimes\delta\right)(g)$.
We conclude that it is enough to show that $\mi\left(\varphi_{0}\otimes\delta\right)$
are $\mathcal{P}'(\nti)$-invariant for homogeneous $\mathcal{P}$
-- but this follows immediately by passing to the limit in the Proposition.
\end{proof}
Since the $\CW$-invariant polynomials on $\lieh_{\C}$ are dense
in the space of smooth functions on the sphere there, it is clear
that $\left\{ \mathcal{P}'(\nti)\right\} $ is precisely the set $\lieh_{\C}^{W_{1}'}$
where $W_{1}'=\Stab_{\CW}(\nti)$. Claim \eqref{enu: m-A-inv} of
the main Theorem is then contained in:
\begin{lem}
Let $W_{1}=\Stab_{W}(\nti)$. Then $\liea_{\C}^{W_{1}}=\lieh_{\C}^{W_{1}'}\cap\liea_{\C}$.\end{lem}
\begin{proof}
The subgroup of a Weyl group fixing a point in $\liea$ (or its dual)
is generated by the root reflections it contains. It follows that
$W_{1}$ is generated by the root reflections $s_{\alpha}$ where
$\alpha\in\rR$ satisfying $B(\alpha,\nti)=0$ (pairing given by the
Killing form on $\lieg$) while $W_{1}'$ is generated by the root
reflections $s_{\alpha'}$ where $\alpha'\in\CR$ satisfies $B'(\alpha',\nti)=0$
(Killing form on $\lieg_{\C}$). Now $B'(\alpha',\nti)=B(\alpha'\upharpoonright_{\liea},\nti)$
since $\lieb$ is orthogonal to $\liea$, where the restrictions are
either roots (or zero). Since $s_{\alpha}$ fixes $H$ iff $\alpha(H)=0$,
while $s_{\alpha'}$ fixes $H$ iff $\alpha'(H)=0$, it follows that:
\[
\lieh_{\C}^{W_{1}'}\cap\liea_{\C}=\cap_{B(\alpha,\nti)=0}\Ker(\alpha)=\liea_{\C}^{W_{1}}\,.
\]

\end{proof}

\section*{Acknowledgments}

This work was done while the author was a member at the Institute
for Advanced Study in Princeton; his stay there was partly supported
by the Institute's NSF grant. He would like to thank Werner Müller
for a useful discussion. The author is currently supported by an NSERC
Discovery Grant.

\bibliographystyle{plain}
\bibliography{que,aut_forms,lie_gps,reference,ergodic_theory}

\end{document}

%% file: mydefs.tex
\usepackage{mathrsfs}
\usepackage{amssymb}
\usepackage{dsfont}
\usepackage{verbatim}
\usepackage{url}

\def\N{\mathbb {N}}
\def\Z{\mathbb {Z}}
\def\Q{\mathbb {Q}}
\def\R{\mathbb {R}}
\def\C{\mathbb {C}}

\newcommand\restrict{\!\upharpoonright}

\def\isom{\simeq}
\def\lap{\triangle}
\def\Lap{\lap}

\def\eqdef{\overset{\text{def}}{=}}
\def\bs{\backslash}

\DeclareMathOperator{\Ad}{Ad}

\def\Ker{\qopname\relax o{Ker}}

\DeclareMathOperator{\Lie}{Lie}

\DeclareMathOperator{\rk}{rk}

\DeclareMathOperator{\Stab}{Stab}

\def\vol{\qopname\relax o{vol}}

\newcommand\norm[1]{\left\Vert {#1} \right\Vert}

\newcommand{\wklim}[1]{\xrightarrow[#1]{\textrm{wk-*}}}

\DeclareRobustCommand
  \rddots{\mathinner{\mkern1mu\raise\p@
    \vbox{\kern7\p@\hbox{.}}\mkern2mu
    \raise4\p@\hbox{.}\mkern2mu\raise7\p@\hbox{.}\mkern1mu}}

\newcommand{\Cc}{C_{\mathrm{c}}}

\newcommand{\Cic}{\Cc^{\infty}}

\newcommand\xhookrightarrow[2][]{\ext@arrow 0062{\hookrightarrowfill@}{#1}{#2}}
\def\hookrightarrowfill@{\arrowfill@\lhook\relbar\rightarrow}

\newcommand\SL{\mathrm{SL}}

\newcommand\SO{\mathrm{SO}}

\newcommand\liea{\mathfrak{a}}
\newcommand\lieb{\mathfrak{b}}

\newcommand\lieg{\mathfrak{g}}
\newcommand\lieh{\mathfrak{h}}

\newcommand\liek{\mathfrak{k}}

\newcommand\liem{\mathfrak{m}}
\newcommand\lien{\mathfrak{n}}

\newcommand\liep{\mathfrak{p}}

\newcommand\lienb{\bar\lien}

\newcommand\ars{\liea_{\R}^{*}}
\newcommand\acs{\liea_{\C}^{*}}

\DeclareMathAlphabet{\mathcal}{OMS}{cmsy}{m}{n}

\newcommand\cC{\mathcal{C}}

\newcommand\cF{\mathcal{F}}

\newcommand\sC{\mathscr{C}}

\def\wrt{w.r.t.\ }